\documentclass[11pt]{amsart}
\usepackage{amssymb,amsmath}
\usepackage{epsfig}

\def\Z{\mathbb{Z}}
\def\Q{\mathbb{Q}}
\def\<{\langle}
\def\>{\rangle}

\newtheorem{theorem}{Theorem}[section]
\newtheorem{lemma}[theorem]{Lemma}
\newtheorem{corollary}[theorem]{Corollary}
\newtheorem{conjecture}{Conjecture}

\newtheorem{thmA}{Theorem}

\newtheorem{corA}[thmA]{Corollary}

\theoremstyle{definition}
 
\newtheorem{definition}[theorem]{Definition}

\begin{document}

\title{Can Dehn surgery yield three connected summands?}
\author{James Howie}
\address{ James Howie\\
Department of Mathematics and Maxwell Institute for Mathematical Sciences\\
Heriot--Watt University\\
Edinburgh EH14 4AS }
\email{ jim@ma.hw.ac.uk}



\maketitle

\begin{abstract}
A consequence of the Cabling Conjecture of Gonzalez-Acu\~{n}a and Short
is that Dehn surgery on a knot in $S^3$ cannot produce a manifold with more
than two connected summands.  In the event that some Dehn surgery produces
a manifold with three or more connected summands, then the surgery parameter
is bounded in terms of the bridge number by a result of Sayari.  Here this
bound is sharpened, providing further evidence in favour of the Cabling Conjecture.
\end{abstract}

\section{Introduction}

The Cabling Conjecture of Gonzalez-Acu\~{n}a and Short \cite{GAS}
asserts that Dehn surgery on a knot in $S^3$ can produce a reducible
$3$-manifold only if the knot is a cable knot and the surgery slope is
that of the cabling annulus.

The Cabling Conjecture is known to hold in many special cases
\cite{HM,HS,LZ,MT,Mu,Sch,Wu}.

If $k$ is the $(p,q)$-cable on a knot $K$, then the cabling annulus
on $k$ has slope $pq$, and the corresponding surgery manifold
$M(k,pq)$ splits as a connected sum $$M(K,p/q)\# L(p,q)$$
\cite{GL}. (Here $L(p,q)$ is a lens space.)  In particular both connected
summands are prime \cite{GL}.  Thus the Cabling Conjecture implies the
weaker conjecture below:

\begin{conjecture}[Two summands conjecture]
Let $k$ be a knot in $S^3$ and $r\in\Q\cup\{\infty\}$ a slope.
Then the Dehn surgery manifold $M(k,r)$ cannot be expressed as a
connected sum of three non-trivial manifolds.
\end{conjecture}

Since any knot group has {\em weight} $1$
(in other words, is the normal closure of a single element),
the same is true for any homomorphic image of a knot group.
Thus the two summands conjecture would follow from the group-theoretic
conjecture below, which remains an open problem.

\begin{conjecture}
A free product of three non-trivial groups has weight
at least $2$.
\end{conjecture}

The best known upper bound for the number of connected summands in $M(k,r)$
is $3$, obtained by combining results of Sayari \cite{Sa1},
Valdez S\'{a}nchez \cite{V} and the author \cite{H}.  These results also show
that, should some $M(k,r)$ have
three connected summands, then two of these must be lens spaces
(necessarily with fundamental groups of coprime orders) and the
third must be a $\Z$-homology sphere. (See \cite{H} for details.)

Suppose that $k$ is a knot in $S^3$ with bridge number $b$,
and that the $3$-manifold $M$ obtained by performing Dehn surgery
on $k$ with surgery parameter $r$ has more than two connected
summands.  It is known from the work of Gordon and Luecke \cite{GL}
that $r$ must be an integer.

If $\ell_1,\ell_2$ are the orders of the fundamental groups
of the lens spaces, then Sayari \cite{Sa2} has proved  that $|r|=\ell_1\ell_2\le (b-1)(b-2)$.

In this paper we shall prove the following inequality.

\begin{thmA}\label{main}
Let $k$ be a knot in $S^3$ with bridge-number $b$.  Suppose that $r$ is a slope
on $k$ such that $M=M(k,r)=M_1\# M_2\# M_3$ where $M_1,M_2$ are
lens spaces and $M_3$ is a homology sphere but not a homotopy sphere.
Then $$|\pi_1(M_1)|+|\pi_1(M_2)|\le b+1.$$
\end{thmA}

As an immediate consequence, we obtain a sharpening of Sayari's inequality.

\begin{corA}\label{maincor}
Under the hypotheses of Theorem \ref{main} we have
$$|r|=|\pi_1(M_1)|\cdot |\pi_1(M_2)|\le b(b+2)/4.$$
\end{corA}

We use the standard techniques of intersection graphs
developed by Scharlemann \cite{Sch1} and by Gordon and Luecke \cite{CGLS,GL2,GL3}.
In \S \ref{graphs} below, we recall the construction of the intersection
graphs in the particular context of this problem.  A key feature
of these is the existence of {\em Scharlemann cycles}, which correspond in
a well-understood way to the lens space summands.  In \S \ref{schar} we
show that, should the inequality $\ell_1+\ell_2\le b+1$ fail, then
we can find, trapped between two Scharlemann cycles, a {\em sandwiched disk}
(see Definition \ref{sand}).  We then show in \S \ref{sd} that sandwiched disks
are impossible, which completes our proof.

\section{The graphs}\label{graphs}

Throughout the remainder of the paper, we assume that the manifold
$M=M(k,r)$ obtained by $r$-Dehn surgery on $k\subset S^3$
is a connected sum of three factors $M_1,M_2,M_3$,
where $M_1$ and $M_2$ are lens spaces while $M_3$
is a (prime) integer homology sphere.  Note that, since $\pi_1(M)$
has weight $1$, the orders $\ell_1,\ell_2$ of
$\pi_1(M_1)$ and $\pi_1(M_2)$ are necessarily coprime.
It follows that the factors
$M_1,M_2,M_3$ are pairwise non-homeomorphic.

An essential embedded sphere $\Sigma\subset M$ necessarily separates,
with one component of $M\smallsetminus\Sigma$ homeomorphic to a punctured $M_s$ and the other
to a punctured $M_t\# M_u$, where $\{s,t,u\}=\{1,2,3\}$.  We will
say that such a $\Sigma$ {\em separates} $M_s$ and $M_t$ (and also
separates $M_s$ and $M_u$).

Let $P_1,P_2$ be disjoint planar surfaces in the exterior $X(k)$ 
of $k$ (the complement of an open regular neighbourhood of $k$ in $S^3$) that extend
to essential spheres $\widehat{P_1},\widehat{P_2}\subset M$ such that
$\widehat{P_i}$ separates $M_i$ and $M_3$.  Assume also
that $P_1,P_2$ have the smallest possible number of boundary components
amongst all such planar surfaces.

A standard argument ensures that we may also choose
$P_1,P_2$ to be disjoint (without increasing the number of boundary
components of either).

Following Gabai \cite[Section 4(A)]{G}, we put $k$ in thin position,
find a level surface $Q$ for $k$ and isotope $P:=P_1\cup P_2$
such that $P$ meets $Q$ transversely, and such that no component of $Q\cap P$ is an
arc that is boundary-parallel in $P$.
(The minimality condition in the definition of $P_1$ and $P_2$ ensures also that
no component of $Q\cap P$ is a boundary-parallel arc in $Q$.)

The number $q$ of boundary components of $Q$ is necessarily even, and is bounded
above by twice the bridge number, $q\le 2b$.  We can complete $Q$ to
a sphere $\widehat{Q}\subset S^3$ by attaching $q$ meridional disks.

We denote the intersection graph of $P_i$ and $Q$ in $\widehat{Q}$
by $G_i$ for $i=1,2$. The ({\em fat}) vertices of $G_i$ are the meridional disks
$\widehat{Q}\smallsetminus Q$, and the edges are the components of $P_i\cap Q$
(some of which may be closed curves rather than arcs).
Each fat vertex contains precisely one point of intersection of $k$ with
$\widehat{Q}$, so a choice of orientation for $k$ and for $Q$ induces an
orientation on the collection of fat vertices -- that is, a partition of
fat vertices into two types, which we call {\em positive} and
{\em negative}.  There are precisely $q/2$ vertices of each type.

Note that the graphs $G_1$ and $G_2$ have the same vertex set
but disjoint edges sets.  Let $G_Q$ denote their union: $G_Q:=G_1\cup G_2$.

Similarly, we denote the intersection graph of $P$ and $Q$ in
$\widehat{P}=\widehat{P_1}\cup\widehat{P_2}$ by $G_P$ (noting that this graph
is the union of two disjoint non-empty subgraphs $G_{P_i}:=G_P\cap\widehat{P_i}$, $i=1,2$,
and hence is not connected).

The edges incident at a vertex $v$ of $G_Q$ are labelled by the boundary
components of $P$.  These labels always occur in the same cyclic order around
$v$ (subject to change of orientation).  We choose a numbering $1,\dots,p$
of $\pi_0(\partial P)$ in such a way that the labels $1,\dots,p$ always occur
in that cyclic order around each vertex of $G_Q$ (without loss of generality,
clockwise for positive vertices and anti-clockwise for negative vertices).

The corner at a vertex $v$ between the edges labelled $x$ and $x+1$ (modulo $p$)
is also given a label: $g_x$ if $v$ is positively oriented, and $g_x^{-1}$ if
$v$ is negatively oriented.  Note that corners are arcs in $\partial X(k)$
with endpoints in $P$.  In the usual set-up for intersection disks, $P$
is connected, and one can interpret the labels $g_x^{\pm 1}$
as elements of $\pi_1(M)$ (relative to a base-point on $P$).  In our context
it is more natural to interpret $g_x^{\pm 1}$ as an element of the path-groupoid
$\Pi=\pi(M,P)$, whose elements are (free) homotopy classes of maps of pairs
from $([0,1],\{0,1\})$ to $(M,P)$.  Thus $\Pi$ is a connected $2$-vertex
groupoid whose vertex groups are isomorphic to $\pi_1(M)$.

Let $T\subset M$ denote the Dehn-filling solid torus, and $k'\subset T$
its core (a knot in $M$).

A {\em Scharlemann cycle} in
$G_i$ is a cycle $C$ bounding a disk-component $\Delta$ of $\widehat{Q}\smallsetminus G_i$
(which we call a {\em Scharlemann disk}), such that each edge of $C$, regarded as an
arc in $P_i$, joins two fixed components of $\partial P_i$ ($x$ and $y$, say).
Thus each edge of $C$ has label $x$ at one end, and $y$ at the other.
Since $x,y$ are consecutive edges of $G_i$ at each vertex of $C$, the edges
of $G_Q\cap\Delta$ between $x$ and $y$ at $v$ belong to $G_{3-i}$ and correspond to
intersection points of $k'$ with $P_{3-i}$.  Since $P_{3-i}$ is separating,
it follows that
$x-y$ is odd, and hence from the {\em parity rule} (see for example \cite[page 386]{GL2})
that all vertices of $C$ have the same orientation.

It is well-known (see for example \cite{CGLS,GL2}) that any Scharlemann cycle
in $G_i$ corresponds to a lens-space summand of $M$.
We have set things up in such a way that this summand is necessarily
isotopic to $M_i$, which leads to the following observation.
(Compare also \cite[Lemma 2.1]{Hoff}, which states a similar
conclusion under slightly different hypotheses.)

\begin{lemma}\label{samelength}
Any Scharlemann cycle in $G_i$ has length $\ell_i:=|\pi_1(M_i)|$.
\end{lemma}

\begin{proof}
Without loss of generality, we may assume that $i=1$.
Let $C$ be a Scharlemann cycle in $G_1$, and $\Delta$ the corresponding
Scharlemann disk.  Assume that $x,y$ are the labels on the edges of $C$.

Following \cite{CGLS,GL2}, we construct a twice punctured
lens space in $M$ as follows.
 The fat vertices of $G_{P_1}$ can be regarded as
meridional slices of the filling solid torus $T$. 
The fat vertices $x$ and $y$ divide $T$ into two $1$-handles,
one of which -- $H$, say -- satisfies
 $\partial\Delta\subset P_1\cup \partial H$.

Then a regular neighbourhood $L$ of $\widehat{P_1}\cup H\cup\Delta$ is a twice-punctured
lens space, with $\pi_1(L)\cong\Z_\ell$, where $\ell$ is the length of $C$.

One component of $\partial L$ is $\widehat{P_1}$.  The second component $\Sigma$ has
precisely two fewer points of intersection with $k'$ than $\widehat{P_1}$.

By the uniqueness of the prime decomposition $M=M_1\# M_2\# M_3$,
$L$ is homeomorphic to a twice-punctured
copy of $M_1$ or of $M_2$.  In the latter case,
$\Sigma$ also separates $M_1$ from $M_3$, which contradicts the minimality
hypothesis on $P_1$.
Hence $L$ is homeomorphic to a twice-punctured copy of $M_1$,
whence   $\ell=\ell_1$ as claimed.
\end{proof}

 More generally,
we have the following essentially well-known result, which is an important tool
in our proof.

Define the $2$-complex $K$ as follows.  $K$ has two vertices, labelled
$1$ and $2$, and $p$ edges, labelled $g_1,\dots,g_p$.  The initial (resp. terminal)
vertex of $g_i$ is $1$ or $2$ depending on whether the vertex $i$ (resp. $i+1$)
of $G_P$ is contained in $P_1$ or in $P_2$.  The $2$-cells of $K$ are in one-to-one
correspondence with the disk-regions of $G_Q$; the attaching map for a $2$-cell
being read off from the corner-labels of the corresponding region of $G_Q$.

\begin{lemma}\label{pres}
Let $K_0$ be a subcomplex of $K$
with $H^1(K_0,\Z)=\{0\}$.
If $K_0$ is connected then $M$ has a connected summand with fundamental group
isomorphic to $\pi_1(K_0)$.  If $K_0$ is disconnected, then $M$ has a connected
summand with fundamental group
isomorphic to $\pi_1(K_0,1)*\pi_1(K_0,2)$.
\end{lemma}

\begin{proof}

The intersection of $\widehat{P}$ with the filling solid torus $T$
is precisely the set of fat vertices of $G_P$, each of which
is a meridional disk in $T$.
These disks divide $T$ into $1$-handles
$H_1,\dots, H_p$, where $H_i$ is the section of $T$ between the fat vertices $i$ and $i+1$
(modulo $p$).

Suppose first that $K_0$ is connected.  Define $K'$ to be the union of the following
subsets of $M$:
\begin{enumerate}
 \item $P_1$ if $K_0$ contains the vertex $1$ of $K$;
\item $P_2$ if $K_0$ contains the vertex $2$ of $K$;
\item the one-handle $H_i$ for each edge $g_i\in K_0$;
\item the disk-region of $G_Q$ corresponding to each $2$-cell of
$K_0$.
\end{enumerate}

It is easy to check that $K'$ is connected, and that $\pi_1(K')\cong\pi_1(K_0)$.
Let $N$ be a regular neighbourhood of $K'$ in $M$

Then $N$ is a compact, connected, orientable $3$-manifold with
$\pi_1(N)\cong \pi_1(K_0)$ and hence $H^1(N,\Z)=\{0\}$.  It follows that
$\partial N$ consists entirely of spheres, by Poincar\'{e} duality.

Capping off each boundary component of $N$ by a ball yields a closed manifold $\widehat{N}$
with $\pi_1(\widehat{N})\cong\pi_1(N)\cong \pi_1(K_0)$, and $\widehat{N}$ is a connected summand of $M$ since
$N\subset M$.

\medskip
Next suppose that $K_0$ is disconnected.  Then $K_0$ contains both vertices $1,2$
of $K$, but no edge from $1$ to $2$.  Choose an edge $g_z$ of $K$ joining $1$ to $2$, and define
$K_1=K_0\cup\{g_z\}$.  Then $K_1$ is connected and $\pi_1(K_1)\cong\pi_1(K_0,1)*\pi_1(K_0,2)$.
Replacing $K_0$ by $K_1$ in the above gives the result.
\end{proof}

\begin{corollary}\label{3lens}
No subcomplex of $K$ has fundamental group which is a free product of three or more
finite cyclic groups.
\end{corollary}

\begin{proof}
Suppose that $K$ has such a subcomplex. Then by Lemma \ref{pres} $M$ has a connected summand
which is the connected sum of three lens spaces.  This contradicts \cite[Corollary 5.3]{H}.
\end{proof}

Finally, the element $R=g_1g_2\dots g_p\in\pi_1(M)$ is a {\em weight element} -- that is,
its normal closure is the whole of $\pi_1(M)$ --
since it is represented by a meridian in $S^3\smallsetminus k$.  This leads to the following
observation, which will be useful later.

\begin{lemma}\label{bigon}
Let $x\in\{1,\dots,p\}$.  Then there is at least one integer $i\in\{1,\dots,(p-2)/2\}$ such
that no $2$-gonal region of $G_Q$ has corners $g_{x+i}$ and $g_{x-i}$ (or $g_{x+i}^{-1}$ and
$g_{x-i}^{-1}$).
\end{lemma}

\begin{proof}
Otherwise we have $g_{x+i}=g_{x-i}^{-1}$ in $\pi_1(M)$ for each $i=1,\dots,(p-2)/2$,
and hence the weight element $W=g_1\dots g_p$ is conjugate to a word of the
form $g_xUg_yU^{-1}$ (where $U=g_{x+1}\cdots g_{x+(p-2)/2}$ and $y=x+\frac{p}{2}$ modulo $p$).
Moreover, $g_x$ is conjugate in $\pi_1(M)$ to an element
of $\pi_1(M_i)$ for $i\in\{1,2,3\}$,
and a similar statement holds for $g_y$.  Hence $W$ belongs to the normal closure in
$\pi_1(M)=\pi_1(M_1)*\pi_1(M_2)*\pi_1(M_3)$ of the free factors containing
conjugates of $g_x$ and $g_y$.  Since all three free factors are non-trivial,  this normal
subgroup is proper, which contradicts
the fact that $W$ is a weight element.
\end{proof}

\section{Analysis of Scharlemann cycles}\label{schar}

By \cite[Proposition 2.8.1]{GL2} there are Scharlemann cycles in $G_1$ and
in $G_2$.  In this section we show that, if $\ell_1+\ell_2$ is big enough, then
these form a configuration we call a {\em sandwiched disk} (which we will show
in the next section to be impossible).  Our next two results should be
compared to \cite[Lemmas 3.2 and 5.3]{Sa2} and \cite[Theorem 2.4]{GL3} respectively,
where the conclusions are similar but the hypotheses slightly different.

\begin{lemma}\label{disjoint}
If $\Delta$ is a Scharlemann disk in bounded by a Scharlemann cycle in $G_1$
(resp. $G_2$) then $\Delta$ contains no edges of $G_2$ (resp. $G_1$).
\end{lemma}

\begin{proof}
Suppose that $\Delta$ is bounded by a Scharlemann cycle $C$ in $G_1$,
and that it contains edges of $G_2$.  By \cite[Proposition 2.8.1]{GL2}
we know that there exists a Scharlemann cycle in $G_2\cap\Delta$.
We will find such a Scharlemann cycle explicitly, and use it to obtain a contradiction.

Recall that $C$ has length $\ell_1$, by Lemma \ref{samelength}.
Let $v_1,\dots,v_{\ell_1}$ denote the vertices of $C$ in cyclic
order.  Each edge of $C$ has labels $x$ and $x+2t+1$, say, which correspond to vertices
in $G_{P_1}$, and the intermediate labels $x+1,\dots,x+2t$ correspond to vertices
of $G_{P_2}$.  (Necessarily, these are even in number and alternating in orientation,
since they correspond to consecutive intersection points of $k'$ with $\widehat{P_2}$
between two consecutive intersection points of $k'$ with $\widehat{P_1}$.)

The graph $Y:=G_2\cap\Delta$ has $\ell_1$ vertices, each of valence $2t$
and each of the same orientation (which we assume to be positive).

If $\ell_1=2$, then every edge of $Y$ joins $v_1$ to $v_2$.  Such an edge has
labels $x+j$ at one end and $x+2t+1-j$ at the other, for some $j$.  The two edges
whose labels are $x+t$ and $x+t+1$ bound a $2$-gonal region, and hence form a Scharlemann
cycle of length $2$.  But then $\ell_1=\ell_2=2$, contradicting the fact that
$\ell_1,\ell_2$ are coprime.

Suppose then that $\ell_1>2$.  There must be a vertex $v_j$ in $C$ that is joined
only to $v_{j-1}$ and $v_{j+1}$ (subscripts modulo $\ell_1$)
by edges of $Y$.  In particular
there are two consecutive vertices of $C$ that are joined by $s\ge t$ edges
of $Y$.  The resulting $s$ $2$-gonal regions of $G_Q\cap\Delta$
give rise to relations $g_{x+j}g_{x+2t-j}=1$ for $0\le j\le s-1$ in the 
path-groupoid $\Pi=\pi(M,P)$.  But all the corners of the Scharlemann disk $\Delta$
have label $h:=g_xg_{x+1}\cdots g_{x+2t}$,
so $h$ has order $\ell_1$ in $\Pi$.  Hence $g_{x+t}$ also has order $\ell_1>2$.
Hence also $s=t$ in the above, for otherwise $g_{x+t}^2=1$ in $\Pi$.

Choose a pair $v_i,v_j$ of vertices of $C$ with $i<j-1$ with $j-i$ minimal
subject to the condition that $v_i,v_j$ are joined by an edge of
$Y$.
Then each pair $(v_i,v_{i+1}),\dots,(v_{j-1},v_j)$ is joined by {\em precisely}
$t$ edges of $Y$, so there is an edge joining $v_i$ and $v_j$
that has labels $x+t$ and $x+t+1$, and this forms part of a Scharlemann cycle
of length $j+1-i$ in $G_2$.  (See Figure 1.)

\begin{center}
\scalebox{0.6}[0.6]{\includegraphics{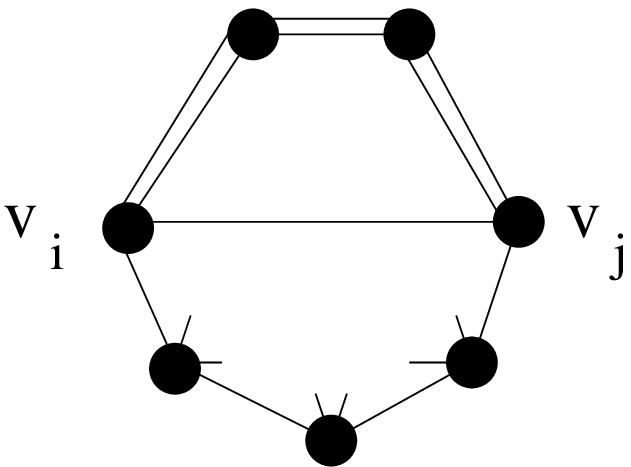}}
\\
Figure 1
\end{center}

 Since $g_{x+t}$ has order
$\ell_1$ in $\Pi$, we deduce that $\ell_1=\ell_2$, which again contradicts
the fact that $\ell_1,\ell_2$ are coprime.
\end{proof}

In particular, if $C$ is a Scharlemann cycle in $G_1$ or $G_2$, then the two
labels appearing on the edges of $C$ are consecutive (modulo $p$): say $x,x+1$.
We call $x$ the {\em label} of $C$.  Note that all the corners of
the corresponding Scharlemann disk have the same label $g_x$ or $g_x^{-1}$.

\begin{corollary}\label{samelabel}
Any two Scharlemann cycles in $G_1$ (respectively, in $G_2$)
have the same label.
\end{corollary}

\begin{proof}
Let $C,C'$ be Scharlemann cycles in $G_1$, bounding Scharlemann disks
$\Delta,\Delta'$ respectively.  By Lemma \ref{disjoint}, $\Delta$ and
$\Delta'$ contain no edges of $G_2$, so are Scharlemann disks of $G_Q$.
By Lemma \ref{samelength} each of $C,C'$ has length $\ell_1$.  Suppose
that $C$ has label $x$ and $C'$ has label $y\ne x$.  Then $K$
has a subcomplex $K_0$ with one vertex $1$, two edges $g_x,g_y$
and two $2$-cells $\Delta,\Delta'$, so that
$$\pi_1(K_0)=\<g_x,g_y|g_x^{\ell_1}=g_y^{\ell_1}=1\>\cong\Z_{\ell_1}*\Z_{\ell_1}.$$
In particular, $\pi_1(K_0)$ has weight $2$, so cannot be isomorphic
to a free factor of $\pi_1(M)$, which contradicts Lemma \ref{pres}.
\end{proof}

\begin{definition}\label{sand}
A {\em sandwiched disk} in $\widehat{Q}$ is a disk $D\subset\widehat{Q}$
such that:
\begin{enumerate}
\item[(a)] $\partial D$ is the union of a subpath $a_1$ of a Scharlemann cycle
$C_1\subset G_1$ and a subpath $a_2$ of a Scharlemann cycle
$C_2\subset G_2$, with $a_1\cap a_2=\partial a_1=\partial a_2$;
\item[(b)] there are no vertices of $G_Q$ in the interior of $D$.
\end{enumerate}
\end{definition}

\begin{lemma}\label{sandw}
If $|\pi_1(M_1)|+|\pi_1(M_2)|>(q+2)/2$, then there exists
a sandwiched disk $D\subset \widehat{Q}$.
\end{lemma}

\begin{proof}
As observed in \cite[p. 551]{Hoff} and \cite[Lemma 6.1]{Sa2}, we know that there are at least
two Scharlemann cycles in $G_1$ -- necessarily
with disjoint sets of vertices, since they have the same label (Lemma \ref{samelabel}).  Similarly there
are at least two Scharlemann cycles in $G_2$ -- again with the same label and hence
with disjoint sets of vertices.

By hypothesis, at least one of $\ell_1=|\pi_1(M_1)|$, $\ell_2:=|\pi_1(M_2)|$
is greater than $q/4$.  Without loss of generality, assume that $\ell_1>q/4$.
Then $G_1$ must contain precisely two Scharlemann cycles, one of each
possible orientation.  Let us call them $C_1^+$ and $C_1^-$,
and let $\Delta_1^\pm$ denote the Scharlemann disks bounded by $C_1^\pm$.

Now let $C_2,C_2'$ denote two disjoint Scharlemann cycles in $G_2$,
and $\Delta_2,\Delta_2'$ the corresponding Scharlemann disks.
Since $\ell_1+\ell_2>q/2$, $C_2$ must intersect $C_1^+$ (if the vertices
of $C_2$ are positive) or $C_1^-$ (if the vertices of $C_2$ are negative).
On the other hand, consideration of vertex orientations shows that
$C_2$ cannot intersect both $C_1^+$ and $C_1^-$.  Similar remarks apply to
$C_2'$.

Now $(C_1^+\cup C_1^-)\cap(C_2\cup C_2')$ consists only of some number
($t$, say) of vertices.

Then $\Delta:=\Delta_1^+\cup \Delta_1^-\cup \Delta_2\cup \Delta_2'$ has precisely two components,
$2\ell_1+2\ell_2-t$ vertices, $2\ell_1+2\ell_2$ edges, and four $2$-cells.
The complement of $\Delta$ in $\widehat{Q}$ thus contains $t-1$ components,
one of which is an annulus and $t-2$ are disks.
But $\widehat{Q}\smallsetminus\Delta$ also contains precisely
$q-2\ell_1-2\ell_2+t$ vertices.  Since $2\ell_1+2\ell_2\ge q+4$, this
number is at most $t-4$.  Hence there are at least two
disk-components of $\widehat{Q}\smallsetminus\Delta$ that contain no vertices of 
$G_Q$.

Moreover, each vertex of $(C_1^+\cup C_1^-)\cap(C_2\cup C_2')$ appears
twice in $\partial(\widehat{Q}\smallsetminus\Delta)$. Each occurrence separates an arc of
$C_1^+\cup C_1^-$ from an arc of $C_c\cup C_2'$ in $\partial(\widehat{Q}\smallsetminus\Delta)$,
so each component of $\partial(\widehat{Q}\smallsetminus\Delta)$ contains an even number of vertices
of $(C_1^+\cup C_1^-)\cap(C_2\cup C_2')$.

The number of boundary components of $\widehat{Q}\smallsetminus\Delta$
is precisely $t$.  If the vertices in $C_2$ and those in $C_2'$ have the same
orientation, then one of $C_1^+,C_1^-$ is a boundary component of $\widehat{Q}\smallsetminus\Delta$
containing no vertices of $(C_1^+\cup C_1^-)\cap(C_2\cup C_2')$.
With that exception, each boundary component of $\widehat{Q}\smallsetminus\Delta$
contains at least at least $2$ vertices of $(C_1^+\cup C_1^-)\cap(C_2\cup C_2')$.

It follows that there is at
least one disk component $D$ of
$\widehat{Q}\smallsetminus\Delta$ whose boundary contains precisely two vertices of
$(C_1^+\cup C_1^-)\cap(C_2\cup C_2')$ and whose interior contains no vertices of $G_Q$.

Any such $D$ is, by definition, a sandwiched disk.
\end{proof}

\section{Analysis of sandwiched disks}\label{sd}

In this section we complete the proof of our upper bound on $|r|$
by showing that sandwiched disks do not exist.  This result holds with
no assumptions on $\ell_1$ or $\ell_2$, so may have wider applications.

We assume throughout that $G_1$, $G_2$ contain Scharlemann cycles of length
$\ell_1,\ell_2$ respectively, with labels $x_1,x_2$ respectively.

\begin{lemma}\label{pover2}
Let $D$ be a sandwiched disk with $\partial D=a_1\cup a_2$,
where $a_1,a_2$ are sub-paths of Scharlemann cycles in $G_1,G_2$ respectively.
Then no two consecutive vertices of $a_1$ (or of $a_2$)
are joined by $p/2$ edges
in $G_Q$.
\end{lemma}

\begin{proof}
Suppose that two vertices of (say) $a_1$ are joined by
$p/2$ edges.  Then there are $2$-gonal regions $D_i$ in $G_Q\cap D$
such that the corner labels of $D_i$ are $g_{x_1+i}$ and $g_{x_1-i}$.
This contradicts Lemma \ref{bigon}.
\end{proof}

\begin{corollary}\label{cor1}
Let $D,a_1,a_2$ be as in Lemma \ref{pover2}.  If two vertices
of $a_1$ (or of $a_2$) are connected by an edge in $G_Q$, then they are
consecutive vertices of $a_1$ (respectively, of $a_2$).
\end{corollary}

\begin{proof}
Let $w_0,\dots,w_t$ be the vertices of $a_1$, in order.  Suppose that
$w_i,w_j$ are joined by an edge in $G_Q$, where $j>i+1$, and that $j-i$
is minimal for such pairs of vertices.  Then $w_{i+1}$ has precisely two
neighbours in $G_Q$: $w_i$ and $w_{i+2}$.  By Lemma \ref{pover2} it is connected
to each by fewer than $p/2$ edges, contradicting the fact that it has
valence $p$.
\end{proof}

\begin{corollary}\label{cor2}
Let $D,a_1,a_2$ be as in Lemma \ref{pover2}.
Each of $a_1,a_2$ has length greater than $1$, and each interior
vertex of $a_1$ (respectively $a_2$) is joined to an interior
vertex of $a_2$ (respectively $a_1$) by an edge of $G_Q\cap D$.
\end{corollary}

\begin{proof}
If $a_1,a_2$ both have length $1$, then every edge of $G_Q\cap D$
joins the two common endpoints $u,v$ of $a_1$ and $a_2$.
Without loss of generality, the edges of $G_Q\cap D$ incident at $u$
have labels $x_1+1,x_1+2,\dots,x_2$, while those incident at
$v$ have labels $x_2+1,x_2+2,\dots,x_1$.
Hence $|x_1-x_2|=p/2$, and
$D$ contains precisely $p/2$
arcs joining $u$ to $v$.
But this contradicts Lemma \ref{pover2}.

If $w$ is an interior vertex of (say) $a_1$, then $w$ has two
neighbours in $\partial D$.  It is joined to each of these by strictly
fewer than $p/2$ arcs, by Lemma \ref{pover2}, and hence is also joined to a third
vertex in $G_Q$.  Since all the edges of $G_Q$ incident at $w$ are
contained in $D$, this third vertex is also in $\partial D$.
By Corollary \ref{cor1} it cannot be a vertex of $a_1$,
so it must be an interior vertex of $a_2$.
\end{proof}

\begin{lemma}\label{ycycles}
Let $D$ be a sandwiched disk in $G_Q$.
Then there are no Scharlemann cycles in $G_Q\cap D$.
\end{lemma}

\begin{proof}
Any Scharlemann cycle $C$ in $G_Q\cap D$ is a Scharlemann cycle in $G_1$ or in
$G_2$, so has label $x_1$ or $x_2$ by Lemma \ref{samelabel}.  Assume without loss of generality that
$C$ has label $x_2$.  For any vertex $v$ of $a_2$, the corner labelled $g_{x_2}$
does not lie in $D$, so the vertices of $C$ are interior vertices of $a_1$.

By Corollary \ref{cor1}, the vertices of $C$ must be pairwise consecutive
vertices of $a_1$, and hence $C$ has length $2$.  Moreover, if $v_1,v_2$ are the
vertices of $C$, then $v_1,v_2$ are connected by edges labelled $x_1+1,\dots x_2$
at one end (say the $v_1$ end), and by edges labelled $x_2+1,\dots x_1$ at the
other ($v_2$) end.  In particular, they are joined by at least $p/2$ edges, contradicting
Lemma \ref{pover2}.
\end{proof}

\begin{corollary}\label{cor3}
If there is a sandwiched disk $D$ in $G_Q$ such that $\partial D=a_1\cup a_2$
where $a_i$ is a subpath of a Scharlemann cycle with label $x_i$, then
$|x_1-x_2|=p/2$.
\end{corollary}

\begin{proof}

Let $a_1\cap a_2=\{u,v\}$.
Without loss of generality, $x_1=p$ and the edges of $G_Q\cap D$ meeting
$u$ are labelled $1,\dots x_2$ at $u$, while those meeting
$v$ are labelled $x_2+1,\dots,p$ at $v$.  If (say) $x_2<p/2$,
then there is a label $y$ with $x_2<y\le p$ such that $y$
does not appear as either label of any edge meeting $u$ that
is contained in $D$.  Consider the subgraph $\Gamma$ of
$G_Q\cap D$ that is obtained by removing $u$ and its incident
edges.  At each vertex of $\Gamma$, the edge labelled $y$
leads to another vertex of $\Gamma$.  Since all vertices of
$\Gamma$ are positive, it follows that $\Gamma$ contains a
great $y$-cycle, and hence a Scharlemann cycle by \cite[Lemma 2.6.2]{CGLS}.
This contradicts Lemma \ref{ycycles}.
\end{proof}

\begin{theorem}\label{nosand}
There are no sandwiched disks in $G_Q$.
\end{theorem}

\begin{proof}
We assume that there is a sandwiched disk $D$ in $G_Q$, and derive
a contradiction.
Suppose that $\partial D=a_1\cup a_2$, where $a_i$ is a subpath of
a Scharlemann cycle $C_i$.
Let $x_i$ be the label of $C_i$.
By Corollary \ref{cor3}, it follows that $|x_1-x_2|=p/2$.

Let $u,v$ denote the common vertices of $a_1,a_2$.
By Corollary \ref{cor1} each of $a_1,a_2$ has length greater than $1$.
Let $s_1,s_2$ be the vertices of $a_1,a_2$ respectively which are
adjacent to $u$,
and let $t_1,t_2$ be the vertices of $a_1,a_2$ respectively which are
adjacent to $v$. (Note that neither of the possibilities $s_1=t_1$, $s_2=t_2$
is excluded at this stage.)

By Corollary \ref{cor1} again, $s_1$ is connected to a vertex of $a_2$ other than $u,v$
by an edge contained in $D$. Similarly, $s_2$ is connected to a vertex of $a_1$ other than $u,v$
by an edge contained in $D$.  These edges cannot cross; hence $s_1$ and $s_2$ are
joined by an edge.  Similarly $t_1$ and $t_2$ are joined by an edge.  Hence each of $u,v$
is incident at a triangular region of $G_Q\cap D$: call them $\Delta_u$ and $\Delta_v$.

Suppose that the edges of $G_Q\cap D$ that are incident at $u$ have labels
$x_1+1,\dots,x_2$ at $u$, and suppose that $i$ of these edges
(namely those with labels $x_1+1,\dots,x_1+i$) are connected to $s_1$.
Then these edges have labels $x_1,x_1-1,\dots,x_1-i+1$ at $s_1$, and together they
bound $i-1$ $2$-gonal faces of $G_Q$, of which the $j$'th has corner labels
$g_{x_1+j}$ and $g_{x_1-j}$.

The remaining $(p-2i)/2$ edges of $G_Q\cap D$ incident at $u$ join $u$ to $s_2$.
They have labels
$x_1+i+1,\dots,x_2$ at $u$, and $x_1-i,\dots,x_2+1$ at $s_2$.  Together they
bound $(p-2i-2)/2$ $2$-gonal regions of $G_Q$, the $j$'th of which has corner
labels $g_{x_2-j}$ and $g_{x_2+j}$.  Thus the triangular region $\Delta_u$ of
$D\cap G_Q$ that is incident at $u$ has corner labels $g_y$ at $u$ and $g_z$
at each of $s_1$ and $s_2$, where $y=x_1+i$ and $z=x_1-i$ (modulo $p$).

We can now perform a similar analysis on the edges of $G_Q\cap D$ that are incident
at $v$.  Note, however, that for all $j\in\{1,\dots,(p-2)/2\}\smallsetminus\{i\}$
there is a $2$-gonal region of $G_Q\cap D$ with corner labels $g_{x_1-j}$ and $g_{x_1+j}$.
By Lemma \ref{bigon} there cannot be a $2$-gonal region of $G_Q\cap D$ with corner
labels $g_{x_1-i}$ and $g_{x_1+i}$.  It follows that there are also precisely $i$
edges joining $v$ to $t_1$, and $(p-2i)/2$ joining $v$ to $t_2$.  The triangular
region $\Delta_v$ of $D\cap G_Q$ that is incident at $v$ then has corner labels
$g_z$ at $v$ and $g_y$ at each of $t_1,t_2$, where $y=x_1+i$ and $z=x_1-i$
as above (see Figure 2).

\begin{center}
\scalebox{0.8}[0.6]{\includegraphics{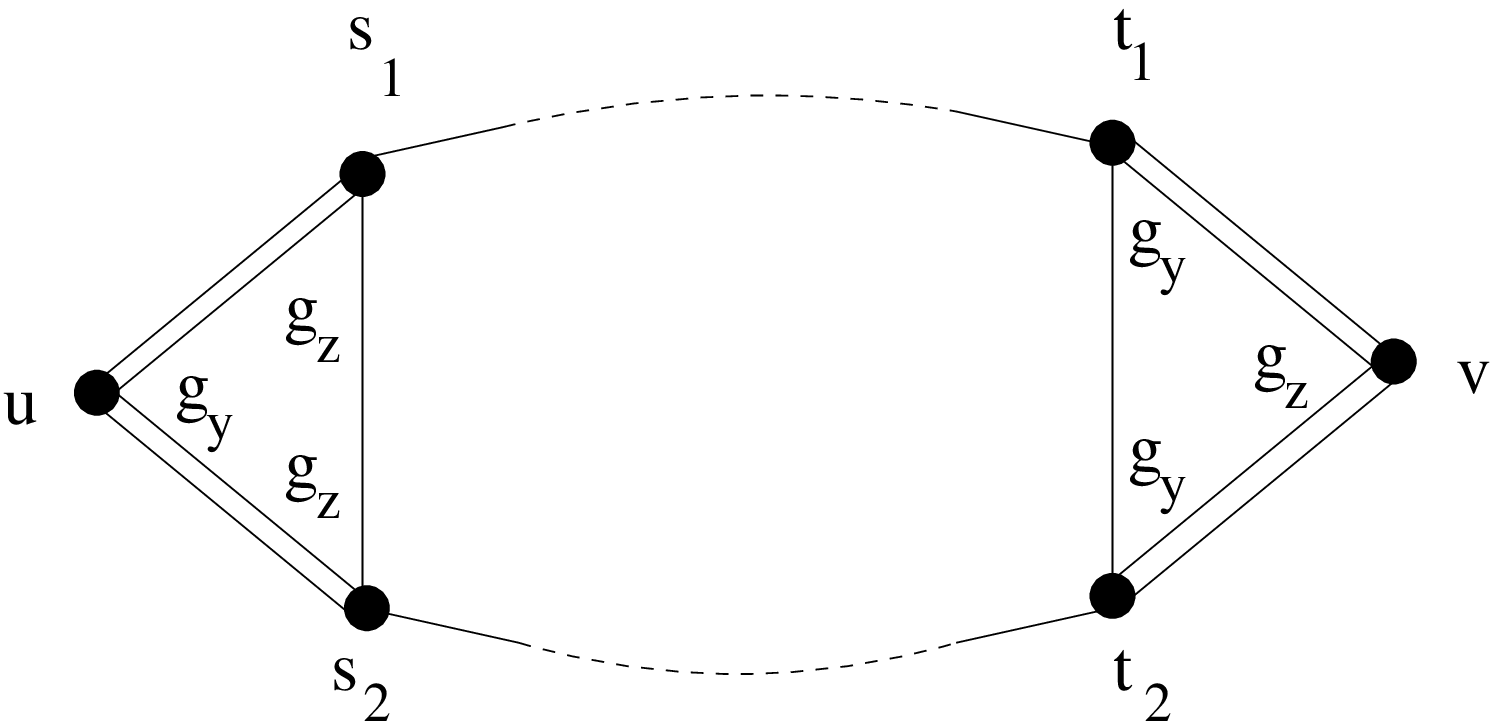}}
\\
Figure 2
\end{center}

Finally, let $K_0$ denote the (disconnected) subcomplex of $K$ with
vertices $\{0,1\}$, edges $\{g_{x_1},g_{x_2},g_y,g_z\}$
and $2$-cells $\{\Delta_1,\Delta_2,\Delta_u,\Delta_v\}$.

Then by Lemma \ref{pres}, $M$ has a connected summand with fundamental group
$$\pi_1(K_0,1)*\pi_1(K_0,2)\cong\<g_{x_1},g_{x_2},g_y,g_z|g_{x_1}^{\ell_1}=g_{x_2}^{\ell_2}=g_yg_z^2=g_zg_y^2=1\>
\cong \Z_{\ell_1}*\Z_{\ell_2}*\Z_3.$$

But this contradicts Corollary \ref{3lens}, which completes the proof.
\end{proof}

\begin{theorem}[= Theorem \ref{main}]\label{inequality}
Let $k$ be a knot in $S^3$ with bridge-number $b$.  Suppose that $r$ is a slope
on $k$ such that $M=M(k,r)=M_1\# M_2\# M_3$ where $M_1,M_2$ are
lens spaces and $M_3$ is a homology sphere but not a homotopy sphere.
Then $$|\pi_1(M_1)|+|\pi_1(M_2)|\le b+1.$$
\end{theorem}

\begin{proof}
As discussed in \S \ref{graphs}, we put $k$ in thin position, and choose a level surface
$Q$ and disjoint planar surfaces $P_1,P_2$ such that
\begin{itemize}
 \item $P_i$ extends to a sphere in $M$ separating $M_i$ from $M_3$, and has
fewest boundary components among all such;
\item no component of $Q\cap P_i$ is a boundary-parallel arc in $Q$ or $P_i$.
\end{itemize}

By Gordon and Luecke \cite{GL}, there are Scharlemann cycles $C_i$ in $G_i$
for $i=1,2$.  Moreover, the Scharlemann cycle $C_i$ has length $\ell_i:=|\pi_1(M_i)|$
and bounds a disk-region $\Delta_i$ of $G_Q$.
If $\ell_1+\ell_2>b+1\ge (q+2)/2$, then
by Lemma \ref{sandw}
there is at least one sandwiched disk $D$ in $G_Q$.
But this contradicts Theorem \ref{nosand}. 

Hence $\ell_1+\ell_2\le b+1$ as claimed.
\end{proof}

\begin{corollary}[= Corollary \ref{maincor}]
With the hypotheses and notation of Theorem \ref{inequality}, we have
$$|r|=|\pi_1(M_1)|\cdot |\pi_1(M_2)| \le \frac{b(b+2)}{4}.$$
\end{corollary}

\begin{proof}
Let $\ell_1=|\pi_1(M_1)|$ and $\ell_2=|\pi_1(M_2)|$.
The equation $|r|=\ell_1\cdot\ell_2$ comes from computing
$|H_1(M,\Z)|$ in two different ways.

Given that $\ell_1,\ell_2$ are distinct positive integers,
the inequality $\ell_1\cdot\ell_2\le b(b+2)/4$ follows easily
from Theorem \ref{inequality}.
\end{proof}

\end{document}